\documentclass[10pt,twoside,leqno]{amsart}
\usepackage{amsmath,amsfonts,amssymb,amsthm}

\theoremstyle{plain}
\newtheorem{theorem}{Theorem}
\newtheorem{proposition}[theorem]{Proposition}
\newtheorem{corollary}[theorem]{Corollary}
\newtheorem{lemma}[theorem]{Lemma}

\theoremstyle{definition}
\newtheorem{definition}[theorem]{Definition}
\newtheorem{example}[theorem]{Example}

\newtheorem{remark}[theorem]{Remark}
\newtheorem{remarks}[theorem]{Remarks}

\newcommand{\define}[1]{{\em #1\/}}

\newcommand{\eps}{\epsilon}
\newcommand{\ve}{\varepsilon}
\newcommand{\G}{\Gamma}

\newcommand{\la}{\lambda}

\newcommand{\de}{\delta}

\newcommand{\cB}{{\mathcal B}}
\newcommand{\cD}{{\mathcal D}}
\newcommand{\cF}{{\mathcal F}}
\newcommand{\fg}{{\mathfrak g}}
\newcommand{\cH}{{\mathcal H}}
\newcommand{\cL}{{\mathcal L}}

\newcommand{\fv}{{\mathfrak v}}

\newcommand{\field}[1]{\mathbb{#1}}
\newcommand{\B}{\field{B}}                  
\newcommand{\N}{\field{N}}                  
\newcommand{\R}{\field{R}}                  
\newcommand{\C}{\field{C}}                  
\renewcommand{\G}{\field{G}}                  

\newcommand{\ang}[1]{\langle #1 \rangle}
\newcommand{\ra}{\rightarrow}
\newcommand{\stm}{\setminus}

\newcommand{\loc}{{\rm loc}}

\newcommand{\inn}{{\it i}}
\newcommand{\out}{{\it o}}

\newcommand{\Rinn}{R^{\inn}}
\newcommand{\Rout}{R^{\out}}
\renewcommand{\L}{\mathcal{L}}
\newcommand{\vfi}{\varphi}
\renewcommand{\la}{\lambda}

\DeclareMathOperator{\BMO}{BMO}
\DeclareMathOperator{\supp}{supp}
\DeclareMathOperator{\diam}{diam}
\DeclareMathOperator{\dist}{dist}

\DeclareMathOperator{\Real}{Re}

\DeclareMathOperator{\spa}{span}

\DeclareMathOperator{\Lip}{Lip}

\newcommand{\restrict}{\begin{picture}(12,12)
                        \put(2,0){\line(1,0){8}}
                        \put(2,0){\line(0,1){8}}
                       \end{picture}}

\makeatletter               
\let\c@equation\c@theorem       
\makeatother                

\numberwithin{theorem}{section}

\title{Removable sets for homogeneous linear PDE in Carnot groups}
\author{Vasilis Chousionis}
\address{Department of Mathematics \\ University of Illinois \\ 1409
  West Green St. \\ Urbana, IL, 61801}
\email{vchous@illinois.edu}
\author{Jeremy T. Tyson}
\address{Department of Mathematics \\ University of Illinois \\ 1409
  West Green St. \\ Urbana, IL, 61801}
\email{tyson@illinois.edu}
\date{\today}
\thanks{JTT supported by NSF grant DMS-1201875.}

\begin{document}
\bibliographystyle{acm}
\maketitle

\begin{abstract}
Let $\cL$ be a homogeneous left invariant differential operator on a Carnot group.
Assume that both $\cL$ and $\cL^t$ are hypoelliptic. We study the removable sets for $\cL$-solutions. We give precise conditions in terms of the Carnot--Carath\'eodory Hausdorff dimension for the removability for $\cL$-solutions under several auxiliary integrability or regularity hypotheses. In some cases, our criteria are sharp on the level of the relevant Hausdorff measure. One of the main ingredients in our proof is the use of novel local self similar tilings in Carnot groups.
\end{abstract}

\section{Introduction}\label{sec:intro}

A fundamental problem in partial differential equations is to understand the size of the sets which are removable for solutions to a given PDE. Such size is typically quantified in terms of both the order of the PDE and the regularity of the solution. Results of this type are well understood in the Euclidean setting, but have received less extensive treatment in more abstract contexts such as sub-Riemannian manifolds and metric spaces. In this paper we study the removability of sets for homogeneous left invariant differential operators on nilpotent stratified Lie groups (also known as Carnot groups).

\

Let us begin with some background and history of the subject in the Euclidean setting. Painlev\'e's problem---to characterize the removability of sets for bounded analytic functions---motivated much work in geometric measure theory, complex analysis and harmonic analysis throughout the twentieth century. The eventual solution to the Painlev\'e problem was obtained by Tolsa \cite{tol:painleve} following extensive work by many people including Ahlfors, Denjoy, Vitushkin, Garnett, Calder\'on, Mattila, Jones, David, Mel'nikov and Verdera. Tolsa's article \cite{tol:icm}, written for the proceedings of the 2006 ICM, is highly recommended for an informative history and survey on the Painlev\'e problem. 

The complete solution to the Painlev\'e problem is quite subtle and relies on intricate geometric properties of the set to be removed related to the {\it Menger curvature}. Among sets $E \subset \C$ of finite Hausdorff $1$-measure, the removable sets for bounded analytic functions are precisely the purely $1$-unrectifiable sets, i.e, sets $E$ which intersect all Lipschitz curves in sets of null $\cH^1$ measure. This was {\it Vitushkin's conjecture}, established by Guy David \cite{dav:lipanalytic}. The same geometric criterion characterizes removable sets for bounded planar harmonic functions, see David and Mattila \cite{dm:lipharm}.

Removability for H\"older continuous harmonic functions (in any dimension) was characterized in an early paper of Carleson \cite{car:holharm}.
Unlike the more difficult case of bounded harmonic functions, there is a precise characterization of H\"older continuous harmonic removability at the level of the Hausdorff measure: {\it a compact set $E \subset \R^n$ is removable for $\delta$-H\"older continuous harmonic functions if and only if $E$ has zero Hausdorff $(n-2+\delta)$-measure.}

In connection with removability problems for functions in the class $\BMO$, we mention Kaufman's paper \cite{kau:bmo} on removability for $\BMO$ analytic functions, and Ishchanov's results \cite{ish:bmo} on removability for $\BMO$ harmonic functions in higher dimensions. 

A comprehensive treatment of removability for solutions of linear partial differential equations in Euclidean spaces of any dimension was given by Harvey and Polking \cite{hp:pde}. They provide conditions on the level of Hausdorff measure for the removability of sets for solutions of homogeneous partial differential operators, under various ancillary integrability, continuity or regularity assumptions. Our results are strongly influenced by the paper of Harvey and Polking.
The survey article \cite{pol:survey} by Polking contains additional results and references to the literature.

The literature on the subject of removability in the Euclidean setting is vast, and the preceding is only a small sample. In particular, we have said nothing about removability for solutions of nonlinear PDE, including the $p$-harmonic or minimal surface equations, or for removability for, e.g., subharmonic functions or quasiregular mappings.

\

In recent years there has been significant emphasis on the extension of classical Euclidean analysis and geometry into more general non-Riemannian spaces, including sub-Riemannian manifolds and more abstract metric measure spaces. While there have been some extensions of the theory of removable sets as described above into such settings, the story appears to be less well understood. Removability for H\"older continuous $p$-harmonic functions in metric spaces was studied by M\"ak\"al\"ainen \cite{mak:removable}; for bounded $p$-harmonic functions see A. Bj\"orn \cite{bjo:removable}.

In this paper we consider the category of nilpotent stratified Lie groups (also known as {\it Carnot groups}) equipped with the Carnot--Carath\'eodory metric. The theory of subelliptic PDE, dating back to the fundamental work of H\"ormander, Folland, Stein, Nagel, Wainger and others, lies at the foundations of this subject. The book by Folland and Stein \cite{fs:hardy} remains an excellent introduction. Nowadays the study of sub-Riemannian spaces is a rich vein within analysis in metric spaces, with numerous applications across both pure and applied mathematics. Carnot groups appear naturally, both as model examples and as local tangent models for regular sub-Riemannian manifolds. In this paper we present an extensive removability theory for solutions to homogeneous left invariant partial differential equations, in the spirit of Harvey and Polking, on arbitrary Carnot groups.

The prototypical example of a Carnot group is the {\it Heisenberg group}. In \cite{cm:lipschitz}, the first author together with Mattila considered the problem of removability for Lipschitz harmonic functions on the Heisenberg group. As in the Euclidean setting the Lipschitz case presents unique technical difficulties. In a follow-up paper \cite{cmt:lipschitz} we will address the problem of Lipschitz removability in general Carnot groups.

We now turn to a description of the results of this paper. We quantify, in terms of the Carnot--Carath\'eodory Hausdorff dimension, the size of removable sets for solutions to homogeneous linear PDE in general Carnot groups, when the equations in question are defined by homogeneous left-invariant partial differential operators satisfying a hypoellipticity criterion. All of our results are sharp at the level of Hausdorff dimension, and some even indicate the sharp value of the corresponding Hausdorff measure for such removability. We study removability for $\BMO$, H\"older, and $L^p_\loc$ solutions. Our main results are Theorem \ref{bmorem}, Theorem \ref{lipderem} and Theorem \ref{lplocrem}. Related results appear in Remark \ref{extraremarks}.


The results in the paper \cite{hp:pde} by Harvey and Polking rely essentially on the construction of a certain smooth partition of unity which has been used from then on by many authors. The resulting {\it Harvey--Polking Lemma} depends on the dyadic tilings of Euclidean space, hence in order to generalize the Harvey--Polking removability results to Carnot groups one must first construct analogous tilings. Although there are many analogues of dyadic cubes in metric spaces (see e.g, Christ \cite{chr:cubes}), none of them can be employed in order to build a useful sub-Riemannian smooth partition of unity. 

One way to realize the usual local dyadic tilings in $\mathbb{R}^n$ is in terms of a certain self similar set of full measure satisfying the open set condition (i.e. the $n$-dimensional cube of sidelength one is divided into $2^n$ cubes of sidelength $1/2$, and so on). We adopt the same approach in general Carnot groups and obtain self-similar tilings which can be employed efficiently for proving a sub-Riemannian Harvey-Polking Lemma.

This paper is organized as follows. In Section \ref{sec:background} we recall preliminary facts about Carnot groups and left invariant operators. In section \ref{sec:hptil} we are concerned with the existence of self-similar sets forming tilings of compact subsets. We use such tilings to develop a sub-Riemannian version of the Harvey--Polking machinery of smooth partitions of unity. In Section \ref{sec:main} we state and prove theorems on the removability of compact sets for $\BMO$ and H\"older continuous solutions of homogeneous left invariant PDE. These results are phrased in terms of the Hausdorff measure in the Carnot--Carath\'eodory metric, and are sharp. In Section \ref{sec:main2} we consider the case of removability of compact sets for $L^p_\loc$ solutions. These results are no longer sharp on the level of the Hausdorff measure, although they remain sharp on the level of Hausdorff dimension. In Remark \ref{extraremarks} we conclude with some related results for other 
 function classes such as the horizontal Sobolev space $W^{k,p}_{H,\loc}$ of functions with $k$th order horizontal derivatives in $L^p_\loc$, and the Folland--Stein space $C^k_H$ of functions with continuous $k$th order horizontal derivatives.

\

\paragraph{\bf Acknowledgments.} The first author would like to thank Joan Mateu for explaining the proof of the $\BMO$ removability result in the Euclidean setting.

\section{Definitions and notation}\label{sec:background}

\subsection{Carnot groups}\label{subsec:carnot}

A \define{Carnot group} is a connected, simply connected, nilpotent
Lie group $\G$ of dimension at least two with graded Lie algebra
$$
\fg=\fv_1\oplus\cdots\oplus \fv_s
$$
so that $[\fv_1,\fv_i]=\fv_{i+1}$ for $i=1,2,\ldots,s-1$ and
$[\fv_1,\fv_s]=0$. The integer $s \ge 1$ is the \define{step}
of $G$. We denote the group law in $\G$ by $\cdot$ and the identity element of $\G$ by $0$. We identify elements of $\fg$ with left invariant vector fields on $\G$ in the usual
manner.

We fix an inner product in $\fv_1$ and let $X_1,\ldots,X_{m_1}$ be an
orthonormal basis for $\fv_1$ relative to this inner product. Using
this basis, we construct the \define{horizontal subbundle} $H\G$ of
the tangent bundle $T\G$ with fibers $H_p\G=\spa\lbrace X_1(p),
\ldots, X_{m_1}(p)\rbrace$, $p\in \G$. A left-invariant vector field $X$
on $\G$ is \define{horizontal} if it is a section of $H\G$. The inner
product on $\fv_1$ defines a left invariant family of inner products
on the fibers of the horizontal subbundle.

We denote by $d$ the \define{Carnot--Carath\'eodory metric} on $\G$,
defined by infimizing the lengths of horizontal paths joining two
fixed points, where the horizontal length is computed using the
aforementioned inner product.

For $t>0$ we define $\delta_t:\fg\to\fg$ by setting $\delta_t(X)=t^i
X$ if $X\in \fv_i$ and extending by linearity. Via conjugation with
the exponential map, $\delta_t$ induces an automorphism of $\G$ which
we also denote by $\delta_t$. Then $(\delta_t)_{t>0}$ is the
one-parameter family of \define{dilations} of $\G$ satisfying
$d(\delta_t(p),\delta_t(q))=t d(p,q)$ for $p,q\in \G$. The Jacobian
determinant of $\delta_t$ (relative to Haar measure) is everywhere
equal to $t^Q$, where
$$
Q=\sum_{i=1}^s i\dim \fv_i
$$
is the \define{homogeneous dimension} of $\G$.
{\bf In this paper, we always assume $Q\ge 3$.}

As a simply connected nilpotent group, $\G$ is diffeomorphic with
$\fg=\R^D$, $D=\sum_{i=1}^s\dim \fv_i$, via the exponential map. The
Haar measure on $\G$ is induced by the exponential map from Lebesgue
measure on $\fg=\R^D$. It also agrees (up to a constant) with the
$Q$-dimensional Hausdorff measure in the metric space $(\G,d)$. In this paper we will denote the Haar measure of a set $E \subset \G$ by $|E|$, and we will write integrals with respect to this measure as $\int_E f(x) \, dx$. We refer the reader to \cite{mont:tour}, \cite{blu:potential} or \cite{cdpt:survey} for further information on Carnot groups and their metric geometry.

\subsection{Left invariant homogeneous operators}\label{subsec:operators}

As usual for $E \subset \G$ we define $C_c^\infty (E)$ to be the set of all $C^\infty$ functions with compact support contained in $E$. For $U$ open we denote by $\cD'(U)$ the space of distributions on $U$ with the usual locally convex topology. Let also $\cD':=\cD'(\G)$. We will denote by $\langle u,f \rangle$ the pairing of a distribution $u \in \cD'(U)$ with a test function $f \in C_c^\infty (U)$.

A distribution $u \in \cD'(U)$ is homogeneous of degree $\lambda$ if $\langle u, f\circ \delta_r \rangle= r^{-Q-\lambda}\langle u,f \rangle$ for all $f \in C_c^\infty$. A distribution which is $C^\infty$ away from $0$ and homogeneous of degree $\lambda-Q$ will be called kernel of type $\la$. A differential operator $\mathcal{L}$ will be called homogeneous of degree $\la$, or $\la$-homogeneous, if $\L (u \circ \delta_r)=r^\la (\L u) \circ \delta_r$ for all $u \in \cD'$. By the Poincar\'e--Birkhoff--Witt theorem \cite{fs:hardy} and homogeneity, the space of $\la$-homogeneous, left invariant differential operators in $\G$ is the linear span of the operators $X_{i_1}\dots X_{i_\la}$ with $i_k=1,\dots m$ and $k=1,\ldots,\la$.

A differential operator $\cL$ on $\G$ is {\it hypoelliptic} if for any open set $U \subset \G$ and any two distributions $u$ and $v$ on $U$ satisfying $\cL u=v$, $v \in C^\infty(U)$ implies $u \in C^\infty(U)$. We call a function $f$ satisfying $\cL f = 0$ an {\it $\cL$-solution}. (By hypoellipticity, such solutions are automatically smooth.)

\begin{example}
The operator $\cL = \sum_{i=1}^m X_i^2$ is known as the {\it sub-Laplacian on $\G$}. The theory of sub-Laplacians on stratified Lie groups has been extensively developed; we refer to Folland \cite{fol:fundamental} and to the books \cite{fs:hardy}, \cite{blu:potential}. Such operators are self-adjoint, $2$-homogeneous and left invariant. By analogy with the Euclidean case, we call solutions of the equation $\cL f = 0$ {\it $\cL$-harmonic} functions.
\end{example}

\begin{definition}
Let $\cL$ be a linear partial differential operator defined on an open set $U$ contained in a domain $\R^D$, and let $E$ be a closed subset of $U$. Let $\cF(U)$ be a class of distributions defined in $U$. The set $E$ is said to be {\it removable for $\cL$-solutions in the class $\cF$}, or {\it removable for $\cF$ $\cL$-solutions},  if whenever $f \in \cF(U)$ satifies $\cL f |_{U\setminus E} = 0$, then $\cL f = 0$.
\end{definition}

\section{Local tilings  and smooth partitions of unity in Carnot groups}\label{sec:hptil}

\subsection{Dyadic decompositions in Carnot groups}\label{subsec:dyadic}

In this section, we construct essentially disjoint partitions of
compact subsets of a Carnot group $\G$, which are analogous to the
classical dyadic decomposition of $\R^n$.

Our starting point is the construction of a fundamental tile $T
\subset \G$, analogous to the Euclidean unit cube. As in the Euclidean
case, $T$ is decomposed into a family of homothetic copies, each of which has size half the size of $T$. We make use of the
following theorem on the existence of self-similar sets in Carnot
groups. Theorem \ref{self-similar-theorem} is a special case of
Proposition 4.14 in~\cite{btw:dimcomp}. In case the group $\G$ is of step two with rational structure constants (for instance, if $\G$ is the Heisenberg group), explicit tilings of this type were constructed by Strichartz \cite{str:self-sim1}, \cite{str:self-sim2}; the latter tilings were further studied in \cite{bhit:horizfractals} and \cite{tys:gcad}.

\begin{theorem}[Existence of self-similar
  tiles]\label{self-similar-theorem}
Let $\G$ be a Carnot group of homogeneous dimension $Q$. There exist
$\tfrac12$-homotheties $f_1,\ldots,f_M \in \G$, $M = 2^Q$, and a
compact set $T \subset \G$ so that $T = \bigcup_{j=1}^M T_j$ where
$T_j=f_j(T)$. Moreover,
\begin{equation}\label{ssth-eqn}
0< |T| < \infty
\end{equation}
and $|T_j \cap T_{j'}| = 0$ whenever $j \ne j'$.
\end{theorem}

A \define{$t$-homothety} of $\G$ is the composition of a fixed left
translation with the dilation $\delta_t$.

Let $W=\{1,\ldots,M\}$. For $m\ge 0$ and $w=w_1\cdots w_m \in W^m$ we
introduce the notation
\begin{equation}\label{fw}
f_w = f_{w_1}\circ\cdots\circ f_{w_m}
\end{equation}
and $T_w=f_w(T)$. We denote by $\cD_m$ the family of sets $T_w$ as $w$
ranges over $W^m$.

\begin{proposition}\label{T-property-1}
The set $T$ is the closure of an open set.
\end{proposition}

In view of Proposition \ref{T-property-1} we may select a point $p \in
T$ and radii $0<\Rinn<\Rout$ so that $B(p,\Rinn) \subset T
\subset B(p,\Rout)$. Fixing $p$, we assume that $\Rinn$ has been
chosen as large as possible and $\Rout$ has been chosen as small as
possible subject to the preceding constraint. Then $\Rout < \diam T$. We call $p$ the
\define{center} of $T$ and $\Rinn$ and $\Rout$ the \define{inner} and
\define{outer radii} of $T$, respectively. For the remainder of the
paper, we fix this data.

\begin{proof}[Proof of Proposition \ref{T-property-1}]
By \cite[Theorem 3.1]{br:doubling}, the assertion in \eqref{ssth-eqn}
implies that the iterated function system (IFS) $\{f_1,\ldots,f_M\}$
satisfies the open set condition (OSC). Choose a bounded open set $O$
so that $f_j(O) \subset O$ for all $j$ and the sets $\{f_j(O)\}$ are
pairwise disjoint. Note that $0<|O|<\infty$. Since $T$
is the invariant set for this IFS, $T \subset \overline{O}$. We claim
that $T = \overline{O}$. If not, $V := O \setminus T$ is a nonempty open set. Since $|O|$ is positive and finite,
$$
\bigl| O \setminus \cup_{j=1}^M f_j(O) \bigr| = 0.
$$
For sufficiently large $m$, there exists $w \in W^m$ so that $f_w(O)
\subset V$. Then the corresponding fixed point $x_w$ for $f_w$ is
contained in $V$. However, since $T$ is the invariant set for the IFS,
$x_w \in T$. This is the desired contradiction.
\end{proof}

For each $w \in W^m$, $m\ge 0$, we define the center of the \define{tile}
$T_w$ to be $p_w = f_w(p)$, and the inner and outer radii of $T_w$ to be
$$
\Rinn_w = 2^{-m}\Rinn \quad \mbox{and} \quad \Rout_w=2^{-m}\Rout,
$$
respectively. We have $B(p_w,\Rinn_w) \subset T_w \subset B(p_w,\Rout_w)$.

The Hausdorff measures and dimensions of bounded subsets of $\G$ can
be computed using a ``dyadic Hausdorff measure'' constructed using the
tiles $\{T_w\}$. By applying a dilation, we may restrict our attention
to subsets of the initial tile $T$. Let $\cD_* = \cup_{m\ge 0} \cD_m$.
For a set $A \subset T$ and $s,\eps>0$, define
$$
\cH_{\cD,\eps}^s(A) = \inf \sum_i (\diam T_{w_i})^s
$$
where the infimum is taken over all coverings of $A$ by tiles $T_{w_i}
\in \cD_*$ with $\diam T_{w_i} < \eps$. Define $\cH_\cD^s(A) =
\lim_{\eps\to0} \cH_{\cD,\eps}^s(A)$. For each $s>0$, $\cH_\cD^s$ is a
Borel measure on $T$ \cite{mat:geometry}. If the intersection of two
tiles has nonempty interior, then one of the tiles is contained inside
the other. As a result we can without loss of generality restrict our
attention in the definition of $\cH_{\cD}^s$ to essentially disjoint
coverings.

\begin{proposition}[Equivalence of Hausdorff and dyadic Hausdorff
  measure]\label{T-property-2}
For each $s>0$ there exists a constant $C=C(s)>0$ so that
$\cH^s(A) \le \cH^s_\cD(A) \le C \cH^s(A)$ for every $A \subset T$.
\end{proposition}

\begin{lemma}\label{T-property-3}
There exists a constant $K>0$ so that for any ball $B(q,r)$ with $r\le
1$ and $m\ge 0$ chosen so that $2^{-m-1} \le r < 2^{-m}$, it holds
that the number of tiles $T_w \in \cD_m$ which intersect $B(q,r)$ is
at most $K$.
\end{lemma}

\begin{corollary}\label{T-property-4}
For each $m\ge 0$, no point of $\G$ lies in more than $K$ of the
tiles in $\cD_m$.
\end{corollary}

\begin{proof}[Proof of Lemma \ref{T-property-3}]
Assume that $T_{w_1},\ldots,T_{w_N}$ is a collection of tiles in
$\cD_m$, each of which intersects $B(q,r)$.
Then $B(p_{w_1},\Rinn_{w_1}),\ldots,B(p_{w_N},\Rinn_{w_N})$ is a
collection of disjoint balls, all contained in the ball $B(q,r+2^{-m}
\diam T)$ and all with radius $2^{-m}\Rinn$. Since
$$
r+2^{-m} \diam T < 2^{-m}(1+2\Rout)
$$
and $\G$ is a doubling metric space, we conclude that $N$ is bounded
above by a constant depending only on $(1+2\Rout)/\Rinn$.
\end{proof}

\begin{proof}[Proof of Proposition \ref{T-property-2}]
It suffices to prove the upper bound.

Since the spherical Hausdorff measure $\cH^s_\cB$ (defined using
coverings by balls) and the usual Hausdorff measure $\cH^s$ are
comparable, it suffices to prove that $\cH^s_{\cD,C\eps}(A) \le C
\bigl( \cH^s_{\cB,\eps}(A) + \eps \bigr)$ for a suitable constant $C$
and for all $\eps>0$.

Assume $A$ is covered by balls $\{B_i\}_i$ with $\diam B_i <
\eps$ and $\sum_i (\diam B_i)^s < \cH^s_{\cB,\eps}(A) + \eps$. By
Lemma \ref{T-property-3}, each $B_i$ is covered by tiles
$\{T_{i,j}:1\le j\le N\}$, where $N$ is bounded independent of $i$ and
$\diam T_{i,j} \simeq \diam B_i$ for all $1\le j \le N$. Then $A$ is
covered by the tiles $\{T_{i,j}\}_{i,j}$ and $\diam T_{i,j} \le C\eps$
for all $i$ and $j$, for some fixed $C$. We obtain
$$
\cH^s_{\cD,C\eps}(A) \le \sum_{i,j} (\diam T_{i,j})^s \le C \sum_i
(\diam B_i)^s \le C \left( \cH^s_{\cB,\eps}(A) + \eps \right).
$$
This completes the proof.
\end{proof}

The dyadic Hausdorff measures $\cH^s_\cD$ are a special case of the
general notion of ``comparable net measures'' introduced by Davies and
Rogers, see for instance \cite[\S2.7]{rog:hausdorff} or \cite[pp.\
64--74]{fal:geometry-of}.

\subsection{Smooth partitions of unity in Carnot groups and a sub-Riemannian Harvey--Polking Lemma}\label{subsec:harvey-polking}

The following lemma is inspired by \cite[Lemma 3.1]{hp:pde}.

\begin{lemma}[Harvey--Polking partition of unity]\label{T-property-7}
Let $\{T_{w_i}:1\le i\le N\}$ be a finite collection of essentially
disjoint tiles, with $T_{w_i} \in \cD_{m(i)}$. For each $i$ there is a
function $\varphi_i \in C^\infty_0(\G)$, supported in
$B(p_{w_i},2\Rout_{w_i})$, so that
$$
\sum_{i=1}^N \varphi_i(q) = 1 \quad \mbox{for all $q \in
  \bigcup_{i=1}^N T_{w_i}$.}
$$
Moreover, for each multi-index $\alpha$ there exists a constant
$C_\alpha>0$ so that
$$
|X_\alpha \varphi_i(x)| \le C_\alpha 2^{m(i)|\alpha|} \quad
\mbox{for all $x \in \G$ and $1\le i\le N$.}
$$
\end{lemma}

Here, for a multi-index $\alpha=(\alpha_1,\ldots,\alpha_\ell) \in
\{1,\ldots,m_1\}^\ell$ (recall that $m_1$ denotes the dimension of the
horizontal layer $\fv_1$ of $\G$) we write
$$
X_\alpha = X_{\alpha_1}X_{\alpha_2}\cdots X_{\alpha_\ell}
$$
and $|\alpha|=\ell$.

\begin{proof}[Proof of Lemma \ref{T-property-7}]
Let $s_i=\diam T_{w_i}$, $i=1,\ldots,N$, and without loss of generality assume that $s_1 \geq s_2 \geq \dots \geq s_N$. Choose $\psi \in C^\infty_0(\G)$ so that $\psi|_T\equiv 1$ and
$\psi|_{\G\setminus B(p,2\Rout)}=0$. For each $i$, define $\psi_i$ by
$$
\psi_i(q) = \psi(f_{w_i}^{-1}(q)),
$$
where $f_w$ is defined in \eqref{fw}. Then $\psi_i \in C^\infty_0(\G)$, $\psi_i|_{T_{w_i}}\equiv 1$ and $\psi_i$ is supported in the ball $B(p_{w_i},2\Rout_{w_i})$. For $1\le i\le N$ define $\varphi_1 = \psi_1$ and $\varphi_{i+1} =
\psi_{i+1}\prod_{k=1}^i(1-\psi_i)$. Then $\varphi_i \in C^\infty_0(\G)$
and $\varphi_i$ is supported in the ball $B(p_{w_i},2\Rout_{w_i})$. Furthermore, $\sum_i \varphi_i = 1$ in the union of the tiles $T_{w_i}$. Let $\theta_i=\sum_{k=1}^i \varphi_i=1-\prod_{k=1}^i (1-\psi_i)$. Since $s_i$ is decreasing it is enough to show
\begin{equation}\label{hptheta}
||X_\alpha \theta_i||_\infty \leq C_\alpha s_i^{-|\alpha|}
\end{equation}
for all $i=1,\dots,N$. Indeed, $X_\alpha \theta_i
=X_\alpha \theta_{i-1} +X_\alpha \varphi_i$, whence (\ref{hptheta}) implies that
$$
||X_\alpha \varphi_i||_\infty \leq ||X_\alpha \theta_i||_\infty +||X_\alpha \theta_{i-1}||_\infty \lesssim s_i^{-|\alpha|}+ s_{i-1}^{-|\alpha|}\lesssim s_i^{-|\alpha|}.
$$
Let
$$
B_\alpha=\left\{ \bar{\beta}=(\beta_1,\dots,\beta_r): \beta_j \ \text{is a multi-index,} \ |\beta_j|\geq 1 \ \text{and} \ \sum_{j=1}^r \beta_j=\alpha \right\}
$$
For every $r$-tuple of multi-indices $\bar{\beta}$ one can assign $i(i-1)\dots(i-r+1)$ $r$-tuples $(\nu_1,\dots,\nu_r)$ in $\{1,\dots,i\}^r$ with all $\nu_j$ distinct. Define the functions
\begin{equation*}
g_{\nu_1, \dots,\nu_r}=
\begin{cases} 0 & \text{if $\nu_j=\nu_k$ for some $j\neq k$}
\\
\prod_{\substack{k\neq \nu_1,\dots,\nu_r \\ k\leq i}}(1-\psi_k) &\text{if all $\nu_j$ are distinct and $r<k$}
\\
-1 &\text{if all $\nu_j$ are distinct and $r=k$.}
\end{cases}
\end{equation*}
Then there exist constants $C_{\beta_1,\dots,\beta_r}$ such that
$$X_\alpha \theta_i=\sum_{(\beta_1,\dots,\beta_r) \in B_\alpha}C_{\beta_1,\dots,\beta_r}\sum_{(\nu_1,\dots,\nu_r)\in\{1,\dots,i\}^r}g_{\nu_1,\dots,\nu_r}(X_{\beta_1} \psi_{\nu_1})\dots(X_{\beta_r} \psi_{\nu_r}) ).$$
Therefore
$$|X_\alpha \theta_i(q)|\le\sum_{(\beta_1,\dots,\beta_r) \in B_\alpha}C_{\beta_1,\dots,\beta_r}\left( \sum_{\nu_1=1}^i |X_{\beta_1} \psi_{\nu_1}(q)|\right)\dots\left(\sum_{\nu_r=1}^i |X_{\beta_r} \psi_{\nu_r}(q)|\right)$$
for all $q \in \G$.

Now consider any of the sums $\sum_{\nu=1}^i |X_{\beta} \psi_{\nu}(q)|$, where $\beta$ is a multi-index and $\nu=1,\dots,i$. Notice that $X_\beta \psi_\nu (q)=0$ if $q \notin B(p_{w_\nu},2 \Rout_{w_v})$. Furthermore for all $i=1,\dots,m$ and $r>0$
\begin{equation*}
\begin{split}
X_i (\psi \circ \delta_r )(q)&= \frac{d}{dt}\psi ((\delta_r(q))\exp(rtX_i))|_{t=0}\\
&=r\frac{d}{dt}\psi ((\delta_r(q))\exp(tX_i))|_{t=0} =r(X_i \psi \circ \delta_r)(q).
\end{split}
\end{equation*}
Hence
$$
X_i \psi_j (q)=X_i(\psi \circ f_j^{-1}(q))=2((X_i \psi)\circ f^{-1}_j)(q)
\quad \mbox{for $j=1,\dots,M$ and $i=1,\dots,m$}
$$
and $||X_\beta \psi_\nu||_\infty \le C s_\nu^{-|\beta|}$ in $B(p_{w_v}, 2 \Rout_{w_v})$. Therefore for all $q \in \G$ we have
$$
\sum_{\nu=1}^i |X_{\beta} \psi_{\nu}(q)|\leq C_\beta \sum_{\substack{\nu \\ s_\nu \geq s_i \ \text{and} \ q \in B(p_{w_\nu},2 \Rout_{w_\nu}) }} |s_\nu|^{-|\beta|}.
$$
Given $q$, as in the proof of Lemma \ref{T-property-3} there exists a finite number $K:=K(\G)$ of tiles $T_{w_1}, \ldots,T_{w_K}$ in $\cD_m$ such that $q \in \bigcap_{l=1}^K B(p_{w_l},2 \Rout_{w_l}).$ Therefore for all $p \in \N$ there exist at most $K$ tiles $T_\nu$ with $\diam T_\nu = 2^p \diam T_i$ such that $q \in B(p_{w_v}, 2 \Rout_{w_v})$. Hence
$$
\sum_{\nu=1}^i |X_{\beta} \psi_{\nu}(q)| \le K C_\beta \sum_{p=0}^\infty (2^p \diam T_i)^{-|\beta|} \lesssim (\diam T_i)^{-|\beta|}
$$
and
$$
|X_\alpha \theta_i(q)| \le\sum_{(\beta_1,\dots,\beta_r) \in B_\alpha}C_{\beta_1,\dots,\beta_r}|s_i|^{-|\beta_1|} \dots|s_i|^{-|\beta_r|}\lesssim |s_i|^{-|\alpha|}.
$$
This completes the proof.
\end{proof}

\section{Removable sets for $\BMO$ and H\"older continuous $\cL$-solutions}\label{sec:main}

Our first theorem characterizes the removable sets for $\BMO$ $\cL$-solutions.

\begin{definition} Let $U$ be an open set in $\G$. A function $f \in L^1_{\loc} (U)$ is in $\BMO(U)$ if any of the two following equivalent conditions is satisfied:
\begin{equation}
\label{bmodef1} \sup_{\substack{x \in \G,\ r>0\\B(x,r)\subset U}}\frac{1}{|B(x,r)|} \int_{B(x,r)}|f(y)-f_{B(x,r)}|\,dy<\infty
\end{equation}
or
\begin{equation}
\label{bmodef2} \sup_{\substack{x \in \G,\ r>0\\B(x,r)\subset U}} \inf_{c \in \mathbb{R}}\frac{1}{|B(x,r)|} \int_{B(x,r)}|f(y)-c|\,dy<\infty.
\end{equation}
As usual $f_{B(x,r)}=\frac{1}{|B(x,r)|} \int_{B(x,r)} f(y) \, dy$ and $\BMO:=\BMO(\G)$.
\end{definition}


\begin{theorem}\label{bmorem}
Fix $1\le\lambda<Q$ and let $\cL$ be a $\la$-homogeneous, left invariant operator on $\G$ such that both $\L$ and $\L^t$ are hypoelliptic. Then a compact set $E$ is removable for $\BMO$ $\cL$-solutions if and only if $\cH^{Q-\la}(E)=0$.
\end{theorem}

Thus, for instance, when $\cL$ is the sub-Laplacian on $\G$, a compact set $E$ is removable for $\BMO$ $\cL$-harmonic functions if and only if $\cH^{Q-2}(E)=0$.

In the Euclidean setting $\R^n$, $n\ge 3$, Theorem \ref{bmorem} is due to Ishchanov \cite{ish:bmo}. Recall that we have assumed $Q\ge 3$ in this paper. The planar case is somewhat different; for results in that setting, see Kaufman \cite{kau:bmo}.

\begin{proof}[Proof of Theorem \ref{bmorem}]
We first suppose that $\cH^{Q-\la}(E)=0$. Let a domain $\Omega \supset E$ and a  function $f \in \BMO(\Omega)$ be given. Let $d_0=\dist(E,\Omega^c)$. Let also $\psi \in C^\infty_0(\Omega)$ and $0<\varepsilon<(d_0/4)^{Q-\lambda}$. By Proposition \ref{T-property-2} and compactness there exists a finite collection $\{T_{w_j}\}_{j=1}^N$ of tiles such that $E \subset \bigcup_{j=1}^N T_{w_j}$ and
\begin{equation*}\label{ecover}
\sum_{j=1}^N (\diam T_{w_j})^{Q-\la}<\varepsilon.
\end{equation*}
Note that $\bigcup_{j=1}^N B(p_{w_j},2\Rout_{w_j}) \subset \Omega$, since $\Rout_{w_j}\le \diam T_{w_j}$.

By Lemma \ref{T-property-7} there exists a family of functions $\{\varphi_j\}_{j=1}^N$ such that $\supp \vfi_j \subset B(p_{w_j}, \Rout_{w_j})$, $\sum_{j=1}^N \vfi_j \equiv 1$ on $\bigcup_{j=1}^N T_{w_j}$ and $\|X_{\alpha} \vfi_j\|_\infty \le C (\diam T_{w_j})^{-|\alpha|}$ for every multi-index $\alpha$ and every $j=1,\ldots,N$. Furthermore, as noted earlier, by the Poincar\'e--Birkhoff--Witt Theorem, $\L$ is a linear combination of the operators $X_{\alpha_l}$ with $|\alpha_l|=\lambda$; therefore we can assume without loss of generality that $\L=X_\alpha$ with $|\alpha|=\la$. Since $\L f=0$ in $\Omega \setminus E$, we conclude
\begin{equation}\begin{split}\label{distest}
\ang{\L f, \psi}&=\ang{\L f, \psi \cdot 1}=\ang{\L f, \sum_{j=1}^N \psi \vfi_j} \\
&= \sum_{j=1}^N \ang{\L f, \psi \vfi_j} =\sum_{j=1}^n \ang{\L (f-c_j), \psi \vfi_j}\\
&=\sum_{j=1}^N (-1)^{|\alpha|} \ang{f-c_j, X_\alpha (\psi \vfi_j)}
=\sum_{j=1}^N (-1)^{|\alpha|} \ang{f-c_j, \sum_{\beta \leq \alpha} \binom{\alpha}{ \beta} X_\beta\psi \, X_{\alpha-\beta}\vfi_j},
\end{split}
\end{equation}
where $c_j=f_{B(p_{w_j}, 2 \Rout_{w_j})}$, $\beta$ is a multi-index with $|\beta|=|\alpha|$ and $\beta \leq \alpha$ denotes that $\beta_i \leq \alpha_i$ for all $i=1,\dots,|\alpha|$. Therefore
\begin{equation*}
\begin{split}
|\ang{\L f, \psi}| &\leq \sum_{j=1}^N | \ang{f-c_j, \psi \, X_\alpha\vfi_j}| + \sum_{j=1}^n \left| \left\langle f-c_j,\sum_{\substack{\beta \leq \alpha \\ \beta \neq \alpha}} \binom{\alpha}{ \beta} X_\beta\psi \, X_{\alpha-\beta}\vfi_j \right\rangle \right| \\
&\lesssim \sum_{j=1}^N \int_{B(p_{w_j},2 \Rout_{w_j})}|f(y)-c_j| \, \|X_\alpha (\vfi_j)\|_\infty \, dy \\
& \qquad \qquad + \sum_{j=1}^N \sum_{\substack{\beta \leq \alpha \\ \beta \neq \alpha}} \int_{B(p_{w_j},2 \Rout_{w_j})}|f(y)-c_j| \, \|X_{\alpha-\beta}\vfi_j\|_\infty \, dy \\
& \lesssim \sum_{j=1}^N (\diam T_{w_j})^{-\lambda} \, |B(p_{w_j}, 2 \Rout_{w_j})| + \sum_{j=1}^N \sum_{\substack{\beta \leq \alpha \\ \beta \neq \alpha}} (\diam T_{w_j})^{|\beta|-\lambda} \, |B(p_{w_j}, 2 \Rout_{w_j})| \\
& \lesssim \sum_{j=1}^N \left( (\diam T_{w_j})^{Q-\lambda} + \sum_{\substack{\beta \leq \alpha \\
\beta \neq \alpha}} (\diam T_{w_j})^{Q-\lambda+|\beta|} \right) 
\lesssim \sum_{j=1}^n (\diam T_{w_j})^{Q-\lambda} < \varepsilon.
\end{split}
\end{equation*}
Since $\varepsilon$ was arbitrary we obtain that $\ang{\cL f, \psi}=0$, which means that $f$ is a distributional solution to $\cL f=0$ in $\Omega$. Therefore by hypoellipticity, $f \in C^\infty(\Omega)$ and $\cL f=0$ in $\Omega$. Hence $E$ is removable for $\BMO$ $\cL$-solutions.

For the other direction it is enough to show that if $\mathcal{H}^{Q-\lambda} (E)>0$, then $E$ is not removable. By Frostman's Lemma there exists a non-trivial Borel measure $\mu$ with $\supp \mu \subset E$, satisfying
\begin{equation}\label{frostman}
\mu (B(x,r)) \le C r^{Q-\lambda} \quad \mbox{for all $x\in \G$ and $r>0$,}
\end{equation}
By \cite[Theorem 2.1]{fol:lie} there is a unique kernel $k$ of type $\lambda$ which is a fundamental solution for~$\L$. Denote by $\|q\|=d(p,0)$ the homogeneous norm associated to the Carnot--Carath\'eodory metric $d$ on $\G$. Since $k$ is continuous on $\Sigma := \{q \in \G:\|q\|=1\}$, there exists $q_1 \in \Sigma$ such that $|k(q)| \le |k(q_1)| $ for all $q \in \Sigma$. Consequently, if $p \in \G \setminus \{0\}$ then
\begin{equation}\label{kernelbound}
|k(p)|=|k(\delta_{\|p\|} \circ \delta_{\|p\|^{-1}}(p))|\leq \|p\|^{\lambda-Q}|k(\delta_{\|p\|^{-1}}(p))|\leq c_1  \|p\|^{\lambda-Q},
\end{equation}
where we set $c_1 := k(q_1)$.

Now let
$$
f=k\ast \mu:=\int k(x \cdot y^{-1}) d \mu(y).
$$
Since $k$ is a fundamental solution of the hypoelliptic operator $\L$ we deduce that $\L f=0$ in~$\G \stm E$.

Furthermore $f \in \BMO$. To see this, first notice that $f \in L^1_ \loc$ with respect to the Haar measure. Let $x_0 \in \G, r>0$ and let $\mu_1 := \mu\restrict B(x_0,2r)$ and $\mu_2 := \mu\restrict B(x_0,2r)^c$ denote the restrictions of $\mu$ to $B(x_0,2r)$ and $B(x_0,2r)^c$ respectively. By (\ref{kernelbound}) and Fubini's theorem,
\begin{equation*}
\begin{split}
\Big| \int_{B(x_0,r)} \int & k(x \cdot y^{-1}) \, d \mu_1 (y) \, dx\Big| \leq \int_{B(x_0,r)} \int_{B(x_0,2r)} \|x \cdot y^{-1}\|^{\la -Q} \, d \mu (y) \, dx \\
&= \int_{B(x_0,2r)}\int_{B(x_0,r)}  \|x \cdot y^{-1}\|^{\la -Q} \, dx \, d \mu (y)\\
&\leq \int_{B(x_0,2r)} \Big( \sum_{j=0}^\infty \int_{B(x,2^{-(j-1)}r)\stm B(x,2^{-j}r)}  \|x \cdot y^{-1}\|^{\la -Q} \, dx \Big) \, d \mu (y)\\
&\lesssim \int_{B(x_0,2r)} r^{\lambda} d \mu (y) \lesssim r^{Q}.
\end{split}
\end{equation*}
Hence
\begin{equation}\label{bmo1}
\int_{B(x_0,r)}|k \ast \mu_1 (x)| \, dx \lesssim r^{Q}.
\end{equation}
For $x\in B(x_0,r)$,
$$
|k\ast\mu_2(x)-k\ast\mu_2(x_0)| \le \int_{B(x_0,2r)^c}|k(x \cdot y^{-1})-k(x_0 \cdot y^{-1})| \, d \mu (y).
$$
Since $k$ agrees with a $C^\infty$, $(\lambda-Q)$-homogeneous function on $\G \setminus \{0\}$,
\begin{equation}\label{folprop1.7}
|k(Y\cdot X)-k(X)|\leq C \|Y\|\|X\|^{\lambda-Q-1} \quad \mbox{for all $\|Y\|\leq \|X\|/2$.}
\end{equation}
This follows exactly as in \cite[Proposition 1.7]{fs:hardy} using the smoothness of the map $y \ra y\cdot x$. Therefore if $x\in B(x_0,r), y \in B(x_0,2r)^c,$ letting $X=x_0\cdot y^{-1},Y=x\cdot x_0^{-1}$ we have that $\|Y\|\leq r \leq \|X\|/2$ and we can apply (\ref{folprop1.7}) to obtain
\begin{equation*}
\begin{split}
|k(x \cdot y^{-1})&-k(x_0 \cdot y^{-1})|=|k(Y\cdot X)-k(X)| \\
&\leq C \|Y\|\|X\|^{\lambda-Q-1}= r \|x_0 \cdot y^{-1}\|^{\lambda-Q-1}.
\end{split}
\end{equation*}
Therefore for $x \in B(x_0,r)$,
\begin{equation*}
\begin{split}
|k\ast\mu_2(x)&-k\ast\mu_2(x_0)| \lesssim r \, \int_{B(x_0,2r)^c}\|x_0 \cdot y^{-1}\|^{\la-Q-1} \, d \mu (y) \\
& = r \sum_{k=1}^\infty \int_{B(x_0,2^{k+1}r)\stm B(x_0,2^{k}r)}\|x_0 \cdot y^{-1}\|^{\la-Q-1} \, d \mu (y) \\
&\leq r \sum_{k=1}^\infty \int_{B(x_0,2^{k+1}r)} (2^k r)^{\la-Q-1} \, d \mu (y) \\
& = r \sum_{k=1}^\infty (2^k r)^{\la-Q-1} \mu (B(x_0,2^{k+1}r))\leq C,
\end{split}
\end{equation*}
and
\begin{equation}\label{bmo2}
\int_{B(x_0,r)}|k\ast\mu_2(x)-k\ast\mu_2(x_0)| \, dx \le C r^{Q}.
\end{equation}
Combining (\ref{bmo1}) and (\ref{bmo2}) we get that for all $x_0 \in \G$ and $r>0$ that
\begin{equation*}
\begin{split}
\int_{B(x_0,r)}|k\ast\mu(x)&-k\ast\mu_2(x)| \, dx = \int_{B(x_0,r)}|k\ast\mu_1(x)+k\ast\mu_2(x)-k\ast\mu_2(x)| \, dx \\
& \le \int_{B(x_0,r)}|k\ast\mu_1(x)| dx+\int_{B(x_0,r)}|k\ast\mu_2(x)-k\ast\mu_2(x)| \, dx \leq C r^{Q}.\\
\end{split}
\end{equation*}
Recalling the equivalent definition (\ref{bmodef2}) for $\BMO$ functions we deduce that $f=k*\mu \in \BMO$.

Now suppose that $\L f=0$. Then $f \in C^{\infty} (\G)$ by hypoellipticity and moreover, by  Geller's Liouville type theorem \cite{gel:liou}, $f$ is a polynomial. By (\ref{kernelbound}),
$$
|f(p)| \le \int_E \|p \cdot q^{-1}\|^{\la-Q} \, d \mu (q)
$$
for all $p \in \G$. Since $E$ is compact this implies that $\lim_{\|p\| \ra \infty} |f(p)|=0$, hence $f \equiv 0$. Since $f$ cannot be constant, we have reached a contradiction. It follows that $E$ is not removable for $\BMO$ $\cL$-solutions.
\end{proof}

Next, we consider the case of H\"older continuous functions. Again, we characterize the removable sets for $\cL$-solutions in this class.
 
\begin{definition}
Let $U \subset \G$ open and $\de \in (0,1)$. A function $f:U \ra  \R$ belongs to $\Lip_\delta (U)$ if there exists some constant $C$ such that
$$
|f(x)-f(y)|\leq C d(x,y)^\de
$$
for all $x,y \in U$.
\end{definition}


\begin{theorem}\label{lipderem}
Fix $1\le\lambda<Q$ and let $\cL$ be a $\la$-homogeneous, left invariant operator on $\G$ such that both $\L$ and $\L^t$ are hypoelliptic. Then a compact set $E$ is removable for $\Lip_\delta$  $\cL$-solutions, $\de \in (0,1)$, if and only if $\cH^{Q-\la+\de}(E)=0$.
\end{theorem}

In particular, if $\cL$ is the sub-Laplacian on $\G$, then a compact set $E$ is removable for $\Lip_\delta$ $\cL$-harmonic functions if and only if $\cH^{Q-2+\delta}(E)=0$.

\begin{proof}[Proof of Theorem \ref{lipderem}]
We first suppose that $\cH^{Q-\la+\de}(E)=0$. The proof in this case is very similar to the proof of the same implication in Theorem \ref{bmorem}. The only difference is that now in (\ref{distest}) we have to choose $c_j=f(p_{w_j})$. 

For the other direction it is enough to show that if $\mathcal{H}^{Q-\lambda+\de} (E)>0$, then $E$ is not removable. Again using Frostman's Lemma we find a non-trivial Borel measure $\mu$ with $\supp \mu \subset E$, satisfying
$$
\mu (B(x,r)) \le C r^{Q-\lambda+\de} \quad \mbox{for all $x\in \G$ and $r>0$,}
$$
If $f=k\ast \mu$ where $k$ is a fundamental solution of the hypoelliptic operator $\L$, it follows that $\L f=0$ in $\G \stm E$.

We are going to show that $f \in \Lip_\de(\G)$. Let $x,z \in \G$. Then 
\begin{equation*}
\begin{split}
|f(x)-f(z)| &= \left| \int k(x\cdot y^{-1}) \, d\mu(y)-\int k(z\cdot y^{-1}) \, d \mu(y)\right| \\ 
&\le \int_{\{\|z\cdot y^{-1}\|>2\|z\cdot x^{-1}\|\}} | k(x\cdot y^{-1})-k(z\cdot y^{-1}) | \, d \mu(y) \\
&\quad\quad\quad + \int_{\{\|z\cdot y^{-1}\|\leq 2\|z\cdot x^{-1}\|\}} | k(x\cdot y^{-1})- k(z\cdot y^{-1})| \, d \mu(y) \\
&:=I_1+I_2.
\end{split}
\end{equation*}
To estimate $I_1$ we use (\ref{folprop1.7}) to obtain
\begin{equation*}
\begin{split}
I_1 \lesssim \|z\cdot x^{-1}\| \int_{\{\|z\cdot y^{-1}\|>2\|z\cdot x^{-1}\|\}} \|z\cdot y^{-1}\|^{\lambda -Q-1} \, d \mu(y).
\end{split}
\end{equation*}
Moreover,
\begin{equation*}
\begin{split}
&\int_{\{\|z\cdot y^{-1}\|>2\|z\cdot x^{-1}\|\}} \|z\cdot y^{-1}\|^{\lambda -Q-1} \, d \mu(y) \\
& \qquad \qquad = \sum_{j=1}^\infty \int_{B(z, 2^{j+1}\|z\cdot x^{-1}\|)\stm B(z, 2^{j}\|z\cdot x^{-1}\|)} \|z\cdot y^{-1}\|^{\lambda-Q-1} \, d \mu(y) \\
& \qquad \qquad \lesssim \sum_{j=1}^\infty \frac{(2^{j+1}\|z\cdot x^{-1}\|)^{Q-\la+\de}}{(2^{j}\|z\cdot x^{-1}\|)^{Q-\la+1}} \lesssim \|z\cdot x^{-1}\|^{\de-1} \sum_{j=1}^\infty (2^{\de-1})^j \lesssim \|z\cdot x^{-1}\|^{\de-1}.
\end{split}
\end{equation*}
Hence $I_1 \lesssim \|z\cdot x^{-1}\|^{\de}$.

For $I_2$, recalling the pseudo triangle inequality $\|p\cdot q\| \leq C(\|p\|+\|g\|)$ for $p,q \in \G$, we see that
$$\{y \in \G:\|z\cdot y^{-1}\|\leq2\|z\cdot x^{-1}\|\} \subset \{y\in \G:\|x\cdot y^{-1}\|\leq 3C\|z\cdot x^{-1}\|\}.$$
Hence
\begin{equation*}
I_2 \leq \int_{\{\|x\cdot y^{-1}\|\leq 3 C\|z\cdot x^{-1}\|\}} |k(x\cdot y^{-1})| \, d \mu (y) +\int_{\{\|z\cdot y^{-1}\|\leq 2\|z\cdot x^{-1}\|\}}| k(z\cdot y^{-1})| \, d \mu (y).
\end{equation*}
Also
\begin{equation*}
\begin{split}
\int_{\{\|x\cdot y^{-1}\|\leq 3 C\|z\cdot x^{-1}\|\}}& |k(x\cdot y^{-1})| \, d \mu(y) \\
&= \sum_{j=0}^\infty \int_{B(x, 2^{-j}3 C\|z\cdot x^{-1}\|)\stm B(x, 2^{-j-1}3 C\|z\cdot x^{-1}\|)} \|x\cdot y^{-1}\|^{\lambda -Q-1} \, d \mu (y) \\
&\lesssim \sum_{j=0}^\infty \frac{(2^{-j}3 C\|z\cdot x^{-1}\|)^{Q-\la+\de}}{(2^{-j-1}3 C\|z\cdot x^{-1}\|)^{Q-\la}} \\
&\lesssim \|z\cdot x^{-1}\|^{\de} \sum_{j=0}^\infty (2^{\de})^{-j} \lesssim \|z\cdot x^{-1}\|^{\de}
\end{split}
\end{equation*}
and in the same way
$$
\int_{\{\|z\cdot y^{-1}\|\leq 2\|z\cdot x^{-1}\|\}}| k(z\cdot y^{-1})| \, d \mu (y) \lesssim \|z\cdot x^{-1}\|^{\de}.
$$
Therefore $I_1+I_2 \lesssim \|z\cdot x^{-1}\|^{\de}$ and $f \in \Lip_\de (\G)$.

Finally arguing as in Theorem \ref{bmorem} and using Geller's Liouville theorem we deduce that $\L f \not\equiv 0$ in $\G$ which implies that $E$ is not removable for $\Lip_\de$ $\cL$-solutions.
\end{proof}

\section{Removable sets for $L^p_\loc$ $\cL$-solutions}\label{sec:main2}

The following lemma is a direct consequence of Lemma \ref{T-property-7}, nevertheless we provide the proof which follows as in \cite{hp:pde}.

\begin{lemma}
\label{harveypolking2}
Let $E \subset \G$ compact, $p \geq 1$ and $\ve>0$. Then there exists $\vfi_\ve \in C_0^\infty(\G)$ such that $\vfi_\ve \equiv 1$ in a neighborhood of $E$, $\supp \vfi_\ve \subset E_\ve:=\{x:\dist(x,E)<\ve\}$ and for all multi-indices $\alpha$ such that $|\alpha| \leq \ell$,
$$\|X_\alpha \vfi_\ve\|_p\lesssim \ve^{\ell-|\alpha|}(\cH^{Q-\ell p}(E)+\ve)^{1/p}.$$
\end{lemma}

\begin{proof} 
We may assume without loss of generality that $\cH^{Q-\ell p}(E)<\infty$. For any $\ve>0$ there exist disjoint tiles $\{T_k\}_{k=1}^N$ which cover $E$ such that $\diam T_k \le \ve$, $\bigcup_{k=1}^N B(p_{w_k},\Rout_{w_k}) \subset E_\ve$ and
\begin{equation}\label{hp2diam}
\sum_{k=1}^N (\diam T_k)^{Q-\ell p} \le C \left( \cH^{Q-\ell p}(E)+\ve \right) \, .
\end{equation}
We can assume that $\diam T_1 \ge \diam T_2 \ge \cdots \ge \diam T_N$. Let $\{\vfi_k\}_{k=1}^N$ be the partition of unity, in the sense of Lemma \ref{T-property-7}, associated to the tiles $\{T_k\}_{k=1}^N$. Then if $\vfi_\ve= \sum_{k=1}^N \vfi_k$ we have that $\supp \vfi_\ve \subset 
\bigcup_{k=1}^N B(p_{w_k},\Rout_{w_k})$ and $\vfi_\ve \equiv 1$ on $\bigcup_{k=1}^N T_k$. 

Let
$$
D_k=B(p_{w_k},\Rout_{w_k}) \setminus \bigcup_{j=k+1}^N B(p_{w_j},\Rout_{w_k}), \qquad \mbox{for $k=1,\ldots,N$.}
$$
The sets $D_k$ are disjoint and $\bigcup_{k=1}^N D_k = \bigcup_{k=1}^N B(p_{w_k},\Rout_{w_k})$. Notice that if $x \in D_k$ then $x \notin \bigcup_{j=k+1}^N B(p_{w_j},\Rout_{w_k})$ which implies that $\vfi_j(x)=0$ for all $x \in D_k$ and $j>k$. Hence 
$$
\vfi_\ve (x)= \sum_{j=1}^k \vfi_j(x)=\theta_k (x)
$$
for all $x \in D_k$, where $\theta_k$ is as in Lemma \ref{T-property-7}, and consequently
$$
|X_\alpha \vfi_\ve (x)| \leq C_\alpha ( \diam T_k)^{-|\alpha|}.
$$
Therefore by (\ref{hp2diam})
\begin{equation*}
\begin{split}
\|X_\alpha \vfi_\ve\|^p_p&= \sum_{k=1}^N \int_{D_k}|X_\alpha \vfi_{\ve}(x)|^p dx= \sum_{k=1}^N \int_{D_k}|X_\alpha \theta_k(x)|^pdx \\
&\lesssim \sum_{k=1}^N \int_{D_k} ( \diam T_k)^{-|\alpha|p}dx \lesssim \sum_{k=1}^N ( \diam T_k)^Q 
( \diam T_k)^{-|\alpha|p} \\
&\lesssim \ve ^{(\ell-|a|)p} \sum_{k=1}^N ( \diam T_k)^{Q-\ell p}\lesssim \ve ^{(\ell-|a|)p} (\cH^{Q-\ell p}(E)+\ve).
\end{split}
\end{equation*}
The proof is complete.
\end{proof}


In our final theorem we consider removable sets for $L^p_\loc$ $\cL$-solutions. Our characterization in the following theorem is sharp on the level of the Hausdorff dimension, but not on the level of the Hausdorff measure.

We denote by $\dim E$ the Hausdorff dimension of a set $E\subset\G$ with respect to the metric~$d$. 

\begin{theorem}\label{lplocrem}
Fix $1\le\lambda<Q$ and let $\cL$ be a $\la$-homogeneous, left invariant operator on $\G$ such that both $\L$ and $\L^t$ are hypoelliptic. Then if $1< p<\infty$:
\begin{enumerate}
\item any compact set $E$ with $\cH^{Q-\la p'}(E)<\infty$ is removable for $L^p_\loc$ $\cL$-solutions, while
\item any compact set $E$ with $\dim E>Q-\la p'$ is not removable for $L^p_\loc$ $\cL$-solutions.
\end{enumerate}
Here as usual $p'$ denotes the conjugate exponent of $p$. If $p=\infty$:
\begin{enumerate}
\item[(3)] any compact set $E$ with $\cH^{Q-\la}(E)=0$ is removable  for $L^\infty_\loc$ $\cL$-solutions, while
\item[(4)] any compact set $E$ with $\dim E>Q-\la$ is not removable for $L^\infty_\loc$ $\cL$-solutions.
\end{enumerate}
Finally,
\begin{enumerate}
\item[(5)] if $E$ is compact with $\cH^{Q-\la}(E)<\infty$, $\Omega \supset E$ is a domain and $f \in L^\infty_\loc(\Omega)$ such that $\L f=0$ in $\Omega \setminus E$, then $\L f$ is a measure supported in $E$.
\end{enumerate}
\end{theorem}

\begin{proof}[Proof of Theorem \ref{lplocrem}]
We are first going to prove (1), (3) and (5) using Lemma \ref{harveypolking2} and following \cite{hp:pde}. Suppose that $\cH^{Q-\la p'}(E)<\infty$. Let a domain $\Omega \supset E$ and a function  $f \in L^p_\loc (\Omega)$ be given. Let also $\psi \in C^\infty_0(\Omega)$ and $\ve>0$. The Poincar\'e--Birkhoff--Witt Theorem implies that $\L$ is a linear combination of the operators $X_{\alpha_l}$ with $|\alpha_l|=\lambda$; therefore we can assume without loss of generality that $\L=X_\alpha$ with $|\alpha|=\la$. Since $\L f=0$ in $\Omega \setminus E$,  if $\vfi_\ve$ is as in Lemma~\ref{harveypolking2},
\begin{equation}
\label{adjoint}
\begin{split}
\ang{\L f, \psi}=\ang{\L f, \psi \vfi_\ve}=(-1)^{\la} \ang{f, X_\alpha (\psi \vfi_\ve)}.
\end{split}
\end{equation}
Since $X_\alpha(\psi \vfi_\ve)=\sum_{\beta \leq \alpha} c_{\alpha,\beta} X_\beta (\psi) \, X_{\alpha-\beta}(\vfi_\ve)$ Lemma \ref{harveypolking2} implies that 
$$\|X_{\alpha-\beta}(\vfi_\ve)\|_p' \lesssim \ve^{\lambda-|\alpha-\beta|}(\cH^{Q-\la p'}(E)+\ve)^{1/p'}$$
and by H\"older's inequality
\begin{equation}
\label{lpest}
\begin{split}
|\ang{f, X_\alpha (\psi \vfi_\ve)}|\leq \int_{E_\ve}|fX_\alpha (\psi \vfi_\ve)| \lesssim \|f\, \chi_{E_{\ve}}\,\|_p \ (\cH^{Q-\la p'}(E)+\ve)^{1/p'}.
\end{split}
\end{equation}
Since $f \in L^p_\loc(\Omega)$ and $m(E)=0$ (as $\dim(E)<Q$) the monotone convergence theorem implies that $\|f\, \chi_{E_{\ve}}\,\|_p \ra 0$ as $\ve \ra 0$. Hence $\ang{\cL f, \psi}=0$, which means that $f$ is a distributional solution to $\cL f=0$ in $\Omega$. Therefore by hypoellipticity, $f \in C^\infty(\Omega)$ and $\cL f=0$ in $\Omega$. Hence $E$ is removable for $L^p_\loc$ $\cL$-solutions. 

Let $f \in L^\infty_\loc (\Omega)$, $\psi$ and $\vfi_\ve$ as in the proof of (1). The proof of (3) follows in the same manner noticing that as in (\ref{lpest}) Lemma \ref{harveypolking2} implies
\begin{equation}
\label{linf}
\begin{split}
|\ang{f, X_\alpha (\psi \vfi_\ve)}| \leq \|f\, \chi_{E_{\ve}}\,\|_\infty \|X_\alpha (\psi \vfi_\ve)\|_1\ \lesssim \cH^{Q-\la}(E)+\ve.
\end{split}
\end{equation}

For the proof of (5), as in (\ref{adjoint}) we get
$$\ang{\L f, \psi}=|-1|^\lambda \ang{f\L  \vfi_\ve, \psi}+|-1|^\lambda\sum_{\beta < \alpha} c_{\alpha,\beta} \ang{f, X_\beta (\psi) \, X_{\alpha-\beta}(\vfi_\ve)}.$$
By Lemma \ref{harveypolking2}
$$|\ang{f, X_\beta (\psi) \, X_{\alpha-\beta}(\vfi_\ve)}|\lesssim \| X_{\alpha-\beta}(\vfi_\ve)\|_1 \lesssim \ve^{\lambda-|\alpha-\beta|}(\cH^{Q-\la}(E)+\ve)$$
hence the distributions
\begin{equation}\label{Deps}
D_\ve:=|-1|^\lambda f \L  \vfi_\ve
\end{equation}
converge weakly to $\L f$ in $\cD'(\Omega)$ as $\ve \ra 0$. Furthermore  $\| \L  \vfi_\ve\|_1 \leq C$ for $\ve>0$ hence the distributions $D_\ve$ are uniformly bounded in $L_1 (\Omega)$ and their weak limit is a measure.

The proof of (2) is as in Carleson's book \cite[\S 7]{car:exceptional}. We provide the details for completeness. By the assumption in (2) there exists $\eta> \max \{ Q-\la p' , 0 \}$ such that $\cH^{\eta}(E)>0$. 
Frostman's Lemma yields a non-trivial Borel measure $\mu$ with $\supp \mu \subset E$, satisfying
$$
\mu (B(x,r)) \le C r^{\eta} \quad \mbox{for all $x\in \G$ and $r>0$.}
$$
For $x \in \G$, let 
$$
u(x)=\int \|x\cdot y^{-1}\|^{\la-Q} \, d\mu (y) .
$$
Let $K \subset \G$ be compact and let $g$ be a nonnegative function such that $\int_\G g(y)^{p'}\,dy=1$. Now if $\theta:=\eta p^{-1}+(Q-\la p')p'^{-1}$ then $Q-\la p'<\theta<\eta$. Define $\vfi:\C \ra \R$ by
$$
\vfi (\zeta) = \int_K \int \frac{g(x)^{p'(1-\zeta)}}{ \|x\cdot y^{-1}\|^{\theta+(Q-\eta)\zeta}} \, d\mu(y) \, dx
$$
and notice that $\vfi$ is analytic. For $t \in \R$, and assuming without loss of generality that $\diam(E)+\diam(K)\leq 1$, we have
\begin{equation*}
\begin{split}
|\vfi(it)|&=\int_K g(x)^{p'} \left(\int \|x\cdot y^{-1}\|^{\theta} \, d \mu (y)\right) , dx\\
&\lesssim \int_K g(x)^{p'} \left(\sum_{j=0}^\infty
\frac{(2^{-j} )^\eta}{(2^{-(j+1)} )^\theta}\right) \, dx\leq M
\end{split}
\end{equation*}
and by Fubini's theorem,
\begin{equation*}
\begin{split}
|\vfi(1+it)|&=\int_K  \int \|x\cdot y^{-1}\|^{\theta+Q-\eta} \, d \mu (y) \, dx\\
&=\int_E \int_K \|x\cdot y^{-1}\|^{\theta+Q-\eta} \, dx \, d \mu (y)
\lesssim \int_E \sum_{j=0}^\infty
\frac{(2^{-j} )^Q}{(2^{-(j+1)} )^{\theta+Q-\eta}} \, dx \leq M.
\end{split}
\end{equation*}
By the maximum modulus principle, $|\vfi (\zeta)|\leq M$ for all $\zeta$ with $0<\Real(\zeta)<1$, and in particular,
$$
\vfi(p^{-1})=\int_K g(x) u(x) \, dx \leq M.
$$
Consequently, if $k$ is the fundamental solution of $\L$, then by (\ref{kernelbound}) we have that
$$
\int_K g(x)\ |k\ast\mu(x)| \, dx \lesssim M.
$$
Finally by duality 
$$
\int_K |k\ast\mu(x)|^p \, dx \lesssim M^p
$$
and $f=k \ast \mu \in L^p_\loc$. As in Theorem \ref{bmorem} Geller's Liouville theorem implies that $\L f \not\equiv 0$ in $\G$, which in turn implies that $E$ is not removable for $L^p_\loc$ $\cL$-solutions. This finishes the proof of (2).
The proof of (4) is much simpler, again using Frostman's lemma and integrating on annuli.
\end{proof}

In the following remarks we assume that $\cL$ is a $\lambda$-homogeneous left invariant partial differential operator on $\G$ such that $\cL$ and $\cL^t$ are hypoelliptic.

\begin{remarks}\label{extraremarks}
(1) Analogous results hold when $L^p_\loc$ is replaced with the Sobolev space $W^{k,p}_{H,\loc}$ consisting of $L^p_\loc$ functions whose iterated horizontal partial distributional derivatives of order at most $k$ exist as functions in $L^p_\loc$. Such results extend to Carnot groups the analogous Euclidean results due to Harvey and Polking \cite[Theorem 4.3]{hp:pde}. For instance, {\it any compact set $E$ with $\cH^{Q-(\la-k) p'}(E)<\infty$ is removable for $W^{k,p}_{H,\loc}$ $\cL$-solutions,} provided $1<p<\infty$.
Moreover, {\it any compact $E$ with $\cH^{Q-\la+k}(E)=0$ is removable for $W^{k,\infty}_{H,\loc}$ $\cL$-solutions.} Finally, {\it if $E$ is compact with $\cH^{Q-\la+k}(E)<\infty$, $\Omega \supset E$ is a domain and $f \in W^{k,\infty}_{H,\loc}(\Omega)$ such that $\L f=0$ in $\Omega \setminus E$, then $\L f$ is a measure supported in $E$.} The proofs are easy variants on the preceding argument; one may retain $k$ derivatives on the function $f$ and transfer only $\lambda-k$ derivatives onto the Harvey--Polking cutoff functions $\varphi_\eps$.

(2) Removability for continuous $\cL$-solutions can also be treated as in \cite[Theorem 4.2(b)]{hp:pde}. One obtains the conclusion that {\it compact sets $E$ with $\cH^{Q-\lambda}$ are removable for continuous $\cL$-solutions.} The rough idea of the proof is as follows. Returning to the proof of part (3) in Theorem \ref{lplocrem}, note that $(\cL\varphi_\eps)$ subconverges weakly to zero in $C(\Omega)'$. Multiplication by a continuous function $f$ is a continuous operation on $C(\Omega)'$, so also $(D_\eps)$ (defined as in \eqref{Deps}) subconverges weakly to zero. Since also $(D_\eps)$ converges weakly to $\cL f$ we conclude that $\cL f =0$ as desired.

(3) As in \cite[Theorem 4.3(d) and Theorem 4.5]{hp:pde} one also obtains results for the spaces $C^k_H$ and $C^{k,\delta}_H$ consisting, respectively, of functions on $\G$ whose iterated horizontal partial derivatives of order at most $k$ are continuous, or of functions in $C^k_H$ whose iterated horizontal partial derivatives of order equal to $k$ are $\delta$-H\"older continuous. In the former case, one has that {\it compact sets $E$ with $\cH^{Q-\lambda+k}(E)<\infty$ are removable for $C^k_H$ $\cL$-solutions.} In the latter case, for $0<\delta<1$, one has that {\it compact sets $E$ with $\cH^{Q-\lambda+k+\delta}(E)=0$ are removable for $C^{k,\delta}_H$ $\cL$-solutions.}

(4) As in the Euclidean setting, the analysis of removability for Lipschitz $\cL$-solutions is more subtle. We will pursue this topic in more detail in a subsequent work \cite{cmt:lipschitz}.
\end{remarks}

We conclude with a few questions for further study.

\begin{remark}
Removability for PDE solutions, as considered in this paper, is traditionally quantified in terms of an appropriate notion of capacity. For Painlev\'e's problem, the relevant capacity is the well known {\it analytic capacity}. In \cite{hp:capacity}, Harvey and Polking develop a capacitary framework which encodes their removability results from \cite{hp:pde}. We have not developed the capacity theory associated to our removability theorems, although such a theory should be straightforward to implement.
\end{remark}

\begin{remark}
Moving even beyond the sub-Riemannian framework, it is natural to inquire about similar removability questions in the general setting of metric measure spaces. Second-order PDE in divergence form can be developed using a weak formulation and the notion of Cheeger differentiability in doubling metric measure spaces admitting a Poincar\'e inequality. See \cite{ch:lipschitz} for Cheeger's differentiation theory, and \cite{bb:p-theory} for the basic machinery of second-order PDE in such spaces. We anticipate that some of our results carry over into this framework.
\end{remark}

\bibstyle{acm}
\bibliography{biblio-carnot-removable}
\end{document}